\documentclass[12pt]{amsart}
\usepackage{amssymb,amsmath}

\newtheorem{thm}{Theorem}[section]

\newtheorem{lem}[thm]{Lemma}
\newtheorem{prop}[thm]{Proposition}
\theoremstyle{definition}

\theoremstyle{remark}
\newtheorem{rem}[thm]{Remark}
\numberwithin{equation}{section}

\newcommand{\eps}{\varepsilon}

\newcommand{\Q}{\mathcal{Q}}
\newcommand{\E}{\mathcal{E}}
\newcommand{\D}{\mathcal{D}}

\newcommand{\ed}{\end {document}}

\newcommand{\p}{\partial}
\newcommand{\pa}{\partial_\alpha}
\newcommand{\pab}{\partial^\beta_\alpha}
\newcommand{\ip}{\mathbf{I-P}}
\newcommand{\ipf}{(\mathbf{I-P})f_\eps }
\newcommand{\pf}{\mathbf{P}f_\eps }
\newcommand{\vgrad}{v\!\cdot\!\nabla\!_x}

\newcommand{\fe}{f_\eps }
\newcommand{\wl}{w^l}
\newcommand{\wlbr}{w^{\{l-|\beta|\}|\gamma|}}
\newcommand{\wT}{w^{2\{l-|\beta|\}|\gamma|}}
\newcommand{\var}{\varepsilon}
\newcommand{\TT}{\mathbb{T}^3}
\newcommand{\aN}{|\alpha|\leq N}
\newcommand{\aaN}{|\alpha|\leq N+1}
\newcommand{\abN}{|\alpha|+|\beta|\leq N}
\newcommand{\grad}{\nabla_{\! x}}
\newcommand{\DIV}{\nabla_{\! x}\!\cdot\!}

\newcommand{\dd}{\mathrm{d}}
\newcommand{\LL}{\mathcal{L}}

\begin{document}

\title[Acoustic Limit]
{Acoustic limit of the Boltzmann equation:
\\ classical solutions}
\author[J. Jang]{Juhi Jang}
\address{Courant Institute of Mathematical Sciences, 251 Mercer Street,
New York, NY 10012}
\email{juhijang@cims.nyu.edu}
\author[N. Jiang]{Ning Jiang}
\address{Courant Institute of Mathematical Sciences, 251 Mercer Street,
New York, NY 10012}
\email{njiang@cims.nyu.edu}

\maketitle

\begin{abstract}
We study the acoustic limit from the Boltzmann equation in the framework of
classical solutions. For a solution $F_\varepsilon=\mu
+\varepsilon \sqrt{\mu}f_\varepsilon$ to the rescaled Boltzmann
equation in the acoustic time scaling
\begin{equation*}
\partial_t F_\varepsilon +\vgrad F_\varepsilon
=\frac{1}{\varepsilon} \Q(F_\varepsilon,F_\varepsilon)\,,
\end{equation*}
inside a periodic box $\mathbb{T}^3$, we establish the
global-in-time uniform energy estimates of $f_\varepsilon$ in
$\varepsilon$ and prove that $f_\varepsilon$ converges strongly to
$f$ whose dynamics is governed by the acoustic system. The collision
kernel $\Q$ includes hard-sphere interaction and inverse-power law
with an angular cutoff.

\end{abstract}

\section{Introduction}
The acoustic system is the linearization about the homogeneous state of
the compressible Euler system. After a suitable choice of units,
in this model the fluid fluctuations $(\rho, u, \theta)$ satisfy
\begin{equation}\label{acoustic-system}
\begin{aligned}
\partial_t \rho+\DIV u=0\,,\qquad &\rho(x,0)=\rho^{0}(x)\,,\\
\partial_t u+\grad(\rho+\theta)=0\,,\qquad & u(x,0)=u^{0}(x)\,,\\
\partial_t \theta+\tfrac{2}{3}\DIV u=0\,,\qquad &\theta(x,0)=\theta^{0}(x)\,.
\end{aligned}
\end{equation}
In this paper, we consider the periodic boundary condition,
i.e $x \in \mathbb{T}^3$.

This is one of the simplest system of fluid dynamical equations
imaginable, being essentially the wave equation. It may be derived
directly from the Boltzman equation as the formal limit of moment
equations for an appropriately scaled family of Boltzmann solutions
as the Knudsen number tends to zero.

The program initiated by Bardos, Golse, and Levermore \cite{BGL2}
was to derive the fluid limits which include incompressible Stokes,
Navier-Stokes, Euler equations, and acoustic system from the {\em
DiPerna-Lions renormalized} solutions. This program has been
developed with great success during the last decade, here we only
mention \cite{BGL2, BGL, GL, GS, LM, LM1, LM2} among others. In
particular, Golse and Saint-Raymond  \cite{GS} justified the first
complete incompressible Navier-Stokes limit from the Boltzmann
equation without any compactness assumption.  On the other hand,
higher order approximations with the unified energy method have been
shown by Guo \cite{G} to give rise to a rigorous passage from the
Boltzmann equations to the Navier-Stokes-Fourier systems beyond the
Navier-Stokes approximations in the framework of classical
solutions.

Surprisingly, the status for rigorously deriving the acoustic system
from DiPerna-Lions solutions of Boltzmann equation is still
incomplete. This is mainly because DiPerna-Lions solutions do not
have some properties which are formally satisfied such as local
conservation laws. In \cite{BGL}, the acoustic limit
 was justified for Maxwell
molecular collisions under some assumption on the amplitude of
fluctuations. The result was significantly improved in \cite{GL} to
a large class of hard potentials and the assumption of the amplitude
of fluctuations was relaxed to the order $\eps^m$ with
$m>\frac{1}{2}$. Recently, the borderline case $m=\frac{1}{2}$ was
covered in \cite{JLM} for soft potentials.

In this paper, we take the first step to establish the acoustic
limit from the Boltzmann equation in the framework of classical
solutions. Working with classical solutions has several advantages
than working with the DiPerna-Lions solutions. For example, the
classical solutions automatically satisfy local conservation laws
and have good regularities; the nonlinear interaction can be
controlled by linear dissipation for small solutions.

We employ the nonlinear energy method developed by Guo \cite{G1, G2,
G} in recent years which has been turned out to be applicable to
other problems, for instance see \cite{J}. We justify the limit for
the case that the amplitude of fluctuation is $\eps$, which is not
being optimal. However, our work has advantages in that we can treat
for a large class of collision kernels in a rather uniform way,
including hard potentials, soft potentials and especially Landau
kernels which were not covered in the framework of the renormalized
solutions. Furthermore, different dissipation mechanisms for
macroscopic parts and microscopic parts in the limit process are
clearly presented by the energy dissipation rate. To our best
knowledge, this is the first global-in-time acoustic limit result in the class of
classical solutions.

The paper is organized as follows: the next section contains the
formulation of the Boltzmann equation for different collision
kernels. Some preliminary lemmas regarding the estimates on the
collision operators are listed in Section 3. Then we give a very
brief formal derivation. Section 5 and 6 are devoted to the energy
estimates.

\section{Formulation and Notations}

Consider the following rescaled Boltzmann equation:
\begin{align}\label{R_B}
\partial_t F_\eps  +\vgrad F_\eps
=\frac{1}{\varepsilon} \Q(F_\eps ,F_\eps )
\end{align}
In this paper, as in \cite{G}, we consider two classes of collision
kernels, the first is given by the standard Boltzmann collision
operator $\Q(G_1,G_2)$:
\begin{align}\label{hard}
&&\Q(G_1,G_2)=\int_{{\mathbb{ R}}^3\times S^2}|u-v|^\gamma B(\theta)
|\{G_1(v^{\prime })G_2(u^{\prime })-G_1(v)G_2(u)\}dud\omega,
\end{align}
where $-3<\gamma\leq 1$, $B(\theta)\leq C|\cos\theta|$, $\text{
}v^{\prime }=v-[(v-u)\cdot \omega ]\omega\text{ }$ and $\text{
}u^{\prime }=u+[(v-u)\cdot \omega ]\omega.$ These collision
operators cover hard-sphere interactions and inverse-power law with
an angular cutoff. The hard potential means $0\leq \gamma\leq 1$,
and the soft potential means $-3<\gamma<0$.

The second class is the Landau collision operator
\begin{equation}\label{landau}
\Q(G_1, G_2)=\sum\limits_{1 \leq i,j \leq 3}\p^i\int\phi_{ij}(v-u)
\{G_1(u)\p^j G_2(v)-G_2(v)\p^j G_1(u)\} \dd u\,,
\end{equation}
where $\p^i=\p_{v_i}$ and
\begin{equation}
\phi_{ij}\equiv \frac{1}{|v|}\left\{\delta_{ij}-\frac{v_i
v_j}{|v|^2} \right\}\,.
\end{equation}

Let
\[
F_\eps =\mu +\varepsilon \sqrt{\mu}f_\eps
\]
be the perturbation around the global Maxwellian
$$\mu=\frac{1}{(2\pi)^{3/2}}e^{-\frac{|v|^2}{2}}.$$ Define
$\mathcal{L}$, the linearized collision operator, as follows
\begin{align}\label{L}
\mathcal{L}g\equiv -\frac{1}{\sqrt{\mu}}\{\Q(\mu,\sqrt{\mu}g)+
\Q(\sqrt{\mu}g,\mu)\},
\end{align}
and the nonlinear collision operator $\Gamma$ as
\begin{align}\label{N}
\Gamma(g,h)=\frac{1}{\sqrt{\mu}} \Q(\sqrt{\mu}g,\sqrt{\mu}h).
\end{align}
The rescaled Boltzmann equation \eqref{R_B} is written in terms of
the perturbation $f_\eps $ as follows:
\begin{align}\label{P_B}
\partial_t f_\eps  +v\cdot\nabla_xf_\eps
+\frac{1}{\varepsilon} \mathcal{L}f_\eps =\Gamma(f_\eps , f_\eps ).
\end{align}
We first recall that the operator $\mathcal{L}\geq0,$ and for
any fixed $(t,x),$ the null space of $\mathcal{L}$ is generated by $[\sqrt{\mu}%
,v\sqrt{\mu},|v|^{2}\sqrt{\mu}]$. For any function $f(t,x,v)$ we
thus can decompose
\[
f=\mathbf{P}f+(\mathbf{I-P})f
\]
where $\mathbf{P}f$ (the hydrodynamic part) is the $L_{v}^{2}$
projection on the null space for $\mathcal{L}$ for given $(t,x).$ We
can further denote
\begin{equation}
\mathbf{P}f=\{\rho_{f}(t,x)+v\cdot u_{f}(t,x)+(\tfrac{|v|^{2}}{2}-\tfrac{3}%
{2})\theta_{f}(t,x)\}\sqrt{\mu}. \label{hfield}%
\end{equation}
Here we define the \textit{hydrodynamic field} of $f$ as
\[
\lbrack\rho_{f}(t,x),u_{f}(t,x),\theta_{f}(t,x)]
\]
which represents the density, velocity and temperature fluctuations
physically.

In order to state  our results precisely, we introduce the following
norms and notations. We use $\langle\cdot\,,\cdot\rangle$ to denote
the standard $L^{2}$ inner product in $\mathbb{R}_{v}^{3},$ while we
use $(\cdot\,,\cdot)$ to denote  the $L^{2}$ inner product either in
$\mathbb{T}^{3}\times\mathbb{R}^{3}$ or in $\mathbb{T}^{3}$ with
corresponding the $L^{2}$ norm $\|\cdot\|.$ We use the standard
notation $H^{s}$ to denote the Sobolev space $W^{s,2}.$ For the
Boltzmann collision operator (\ref{hard}), we define the collision
frequency as
\begin{equation}
\nu (v)\equiv \int_{\mathbb{R}^3} |v-v'|^\gamma\mu(v')dv',
\label{nu}
\end{equation}
which behaves like $|v|^\gamma$ as $|v|\rightarrow\infty.$ It is
natural to define the following weighted $L^{2}$ norm to
characterize the dissipation rate.
\[
|g|_{\nu}^{2} \equiv\int_{\mathbb{R}^{3}}g^{2}(v)\nu(v)dv,\;\;\;
\;\|g\|_{\nu}^{2}   \equiv\int_{\mathbb{T}^{3}\times\mathbb{R}^{3}}%
g^{2}(x,v)\nu(v)\dd v\dd x.
\]
For the Landau operator \eqref{landau}. let
\begin{equation}
\sigma_{ij}(v)=\int_{\mathbb{R}^3}\frac{1}{|v-u|}\left\{
\delta_{ij}- \frac{(v-u)_i(v-u)_j}{|v-u|^2}\right\} \mu(u)\,\dd u\,.
\end{equation}
The natural norms are given by the $\sigma\mbox{-}$norm
\begin{equation*}
\begin{split}
|g|^2_\sigma &\equiv \sum\limits_{1 \leq i,j \leq 3}
\int_\mathbb{R}^3
\{ \sigma_{ij}\p^i g\p^j g +\sigma_{ij} v^i v^j g^2\}\,\dd v\,,\\
\|g\|^2_\sigma &\equiv \sum\limits_{1 \leq i,j \leq 3}
\int_{\mathbb{R}^3 \times \mathbb{T}^3} \{ \sigma_{ij}\p^i g\p^j g
+\sigma_{ij} v^i v^j g^2\}\,\dd v\dd x\,.
\end{split}
\end{equation*}
We also use a unified notation for the dissipation as $|g|_D$ and
$\|g\|_D$ to denote either $|g|_\nu$ or $|g|_\sigma$, $\|g\|_\nu$ or
$\|g\|_\sigma$ respectively. Let the weight function $w(v)$ be
\[
 w(v)\equiv |(1+|v|^2)^{\frac{1}{2}}.
\]
For both Boltzmann and Landau kernels we have
\begin{equation}
\|w^{-3/2}g\|\leq C \|g\|_D\,.
\end{equation}
See \cite{G} for the details.

In order to be consistent with the hydrodynamic equations, we define
\begin{equation}
\partial_{\alpha}^{\beta}=\partial_{x_{1}}^{\alpha_{1}}\partial_{x_{2}%
}^{\alpha_{2}}\partial_{x_{3}}^{\alpha_{3}}\partial_{v_{1}}^{\beta_{1}%
}\partial_{v_{2}}^{\beta_{2}}\partial_{v_{3}}^{\beta_{3}} \label{derivative}%
\end{equation}
where $\alpha=[\alpha_{1},\alpha_{2},\alpha_{3}]$ is related to the
space derivatives, while $\beta=[\beta_{1},\beta_{2},\beta_{3}]$ is
related to the velocity derivatives.

We now define instant energy functionals and the dissipation
rate.\\

\textbf{Definition 1 (Instant Energy) }\textit{For
}$N\geq8$,\textit{\ for some constant }$C>0,$\textit{\ an instant
energy functional }$\mathcal{E}_{N,l}(f)(t)\equiv
\mathcal{E}_{N,l}(t)$ \textit{satisfies:}\\

 (i) for hard potentials
with $0\leq \gamma\leq 1$ in \eqref{hard}
\begin{equation}\label{E}%
\frac{1}{C}\mathcal{E}_{N,l}(t)\leq\sum_{|\alpha|\leq
N+1}\|\p_\alpha f\|^2
+\sum_{|\alpha|+|\beta|\leq N}%
\|\wl\pab f\|^{2}\leq C\mathcal{E}_{N,l}(t)\,;
\end{equation}

(ii) for soft potentials with $-3<\gamma< 0$ in \eqref{hard}
\begin{equation}\label{E2}%
\frac{1}{C}\mathcal{E}_{N,l}(t)\leq
\sum_{|\alpha|\leq N+1}\|\p_\alpha f\|^2+\sum_{|\alpha|+|\beta|\leq N}%
\|w^{\{l-|\beta|\}|\gamma|}\pab f\|^{2}\leq C\mathcal{E}_{N,l}(t)\,;
\end{equation}

(iii) for the Landau kernel \eqref{landau},
\begin{equation}\label{E3}%
\frac{1}{C}\mathcal{E}_{N,l}(t)\leq
\sum_{|\alpha|\leq N+1}\|\p_\alpha f\|^2+\sum_{|\alpha|+|\beta|\leq N}%
\|w^{l-|\beta|}\pab f\|^{2}\leq C\mathcal{E}_{N,l}(t).\,,
\end{equation}
for all functions $f(t,x,v)$.\\

\textbf{Definition 2 (Dissipation Rate) }\textit{For
}$N\geq8$,\textit{\ the dissipation rate
}$\mathcal{D}_{N}(t)$\textit{\ is defined as}\\

(i) for hard potentials with $0\leq \gamma\leq 1$ in \eqref{hard}
\begin{equation}\label{D}
\begin{split}
\mathcal{D}_{N,l}(t)=\sum_{|\alpha|\leq N+1}\left(\eps\|
\partial_{\alpha}\mathbf{P}f\|^{2}(t)+\frac{1}{\eps}\|\p_\alpha
(\mathbf{I-P})f\|^2_\nu\right)
\\+\frac{1}{\varepsilon} \sum_{|\alpha|+|\beta|\leq
N}\|\wl\pab(\mathbf{I-P})f \|_{\nu}^{2};
\end{split}
\end{equation}

(ii) for soft potentials with $-3<\gamma< 0$ in \eqref{hard}
\begin{equation}\label{D2}
\begin{split}
 \mathcal{D}_{N,l}(t)=\sum_{|\alpha|\leq N+1}\left(\eps\|
\partial_{\alpha}\mathbf{P}f\|^{2}(t)+\frac{1}{\eps}\|\p_\alpha
(\mathbf{I-P})f\|^2_\nu\right)
\\+\frac{1}{\varepsilon} \sum_{|\alpha|+|\beta|\leq
N}\|w^{\{l-|\beta|\}|\gamma|}\pab(\mathbf{I-P})f \|_{\nu}^{2}.
\end{split}
\end{equation}

(iii) for the Landau kernel \eqref{landau},
\begin{equation}\label{D3}
\begin{split}
 \mathcal{D}_{N,l}(t)=\sum_{|\alpha|\leq N+1}\left(\eps\|
\partial_{\alpha}\mathbf{P}f\|^{2}(t)+\frac{1}{\eps}\|\p_\alpha
(\mathbf{I-P})f\|^2_\nu\right)
\\+\frac{1}{\varepsilon} \sum_{|\alpha|+|\beta|\leq
N}\|w^{l-|\beta|}\pab(\mathbf{I-P})f \|_{\nu}^{2}.
\end{split}
\end{equation}\

Both the instant energy and the dissipation rate are carefully
designed to capture the structure of the rescaled Boltzmann equation
\eqref{R_B} in the acoustic regime. For soft potentials,
$\mathcal{E}_{N,l}$ and $\mathcal{D}_{N,l}$ involve a weight
function in $v$ which depends on the number of velocity derivatives
$\partial^\beta$. This is designed to control the velocity
derivatives for the streaming terms $v\cdot\nabla_x$ by a weak
dissipation rate as proposed in \cite{G}. In particular, the
dissipation rates in \eqref{D}, \eqref{D2}, \eqref{D3} in which the
hydrodynamic part has $\varepsilon$ scale
 reflect that we do not observe the dissipation in the limit, which is
exactly the case of the acoustic system.

We state the main result of this article.

\begin{thm}
\label{Acoustic} Let $N\geq 8$. Let $0<\varepsilon\leq \frac{1}{4}$
be
 given.  Suppose $f_\eps  (0,x,v)=f_0^\varepsilon
(x,v)$ satisfies the mass, momentum, and energy conservation laws
\begin{equation}
 (f_0^\varepsilon,[1,v,|v|^2]\sqrt{\mu})=0,
\end{equation}
and $F_\eps  (0,x,v)=\mu+\varepsilon f_0^\varepsilon (x,v)\geq 0$.
If $\;\mathcal{E}_{N,l}(f^{\varepsilon})(0)$ is sufficiently small,
then there exists a unique global-in-time solution $f_\eps (t,x,v)$
to \eqref{P_B}, and moreover there exists an instant energy
functional $\mathcal{E}_{N,l}(f^{\varepsilon})(t)$ such that
\begin{equation}
\frac{d}{dt}\mathcal{E}_{N,l}(f^{\varepsilon})(t)+\mathcal{D}_{N,l}%
(f^{\varepsilon})(t)\leq0. \label{en}%
\end{equation}
In particular, we have the following global energy bound:
\begin{equation}
\sup_{0\leq t\leq\infty}\mathcal{E}_{N,l}(f^{\varepsilon})(t)\leq
\mathcal{E}_{N,l}(f^{\varepsilon})(0). \label{enbound}%
\end{equation}
\end{thm}

\begin{rem}
 The global existence of solutions $f_\eps $ to \eqref{P_B} follows from
the a priori
 global energy bound \eqref{enbound} by rather standard method.
In this article, we focus on proving the uniform bound.
\end{rem}

\begin{rem} Note that due to the weak dissipation \eqref{D}, we cannot deduce
the time decay estimate from the energy inequality \eqref{en} unlike
the incompressible Navier-Stokes-Fourier case in \cite{G, J}.
Indeed, physically, we do not expect any time decay of our instant
energy $\mathcal{E}_{N,l}(f^{\varepsilon})(t)$, since the acoustic
system preserves the initial energy for all time. See Lemma
\ref{regularity}.
\end{rem}

\section{Basic estimates of collision operators}

In this section, we sum up some basic estimates of collision
operators for various kernels considered in this paper. The proofs
can be found in \cite{G}. The following is the coercivity of
$\mathcal{L}$.

\begin{lem} There exists $\delta >0$ such that for any $f\in L^2
(\mathbb{R}_v^3)$
\begin{align}\label{coer}
\langle \mathcal{L}f,f \rangle \geq  \delta |(\mathbf{I-P})f|_\nu^2.
\end{align}
\end{lem}

\begin{lem}\label{Hard-Est}
For hard potential with $\gamma\geq 0$, there exits $C_{|\beta|}$,
$C>0$ such that
\begin{equation}\label{linear}
(w^{2l}\pab\mathcal{L}f\,,\pab f)\geq \frac{1}{2}\|w^l\pab
f\|^2_\nu-C_{|\beta|}\|f\|^2_\nu\,,
\end{equation}
\begin{equation}\label{nonlinear}
(\partial_\alpha^\beta\Gamma(f, g),\partial_\alpha^\beta h)\leq C
\{\|\wl\partial_{\alpha_1}^{\beta_1}
f\|\cdot\|\wl\partial_{\alpha_2}^{\beta_2} g\|_\nu
+\|\wl\partial_{\alpha_1}^{\beta_1}
g\|\cdot\|\wl\partial_{\alpha_2}^{\beta_2}
f\|_\nu\}\|\wl\partial_\alpha^\beta h\|_\nu\,.
\end{equation}
where $l\geq 0$, and summation is for $|\alpha|+|\beta|\leq N$ with
$\beta_1+\beta_2\leq\beta$ and $\alpha_2\leq \alpha$ componentwise.
\end{lem}
\begin{lem}\label{Soft-Est}
For the inverse power law with $-3 <\gamma <0$, for any $l \geq 0$,
there exist $C_{|\beta|}\,, C >0$ such that
\begin{equation}\label{linear-soft}
(w^{\{2l-2|\beta|\}|\gamma|}\pab \LL f\,,\pab f) \geq
\frac{1}{2}\|\wlbr\pab f\|^2_\nu -C_{|\beta|}\|f\|^2_\nu\,,
\end{equation}
\begin{equation}\label{nonlinear-soft}
\begin{split}
(w^{\{2l-2|\beta|\}|\gamma|}\pab\Gamma(f,g)\,,\pab h) &\leq C\{\|w^{\{l-|\beta_1|\}|\gamma|}\p^{\beta_1}_{\alpha_1} f\|\cdot\|w^{\{l-|\beta_2|\}|\gamma|}\p^{\beta_2}_{\alpha_2} g\|_\nu\\
&+\|w^{\{l-|\beta_1|\}|\gamma|}\p^{\beta_1}_{\alpha_1} g\|\cdot\|w^{\{l-|\beta_2|\}|\gamma|}\p^{\beta_2}_{\alpha_2} f\|_\nu\}\\
&\times \|\wlbr\pab h\|_\nu\,,
\end{split}
\end{equation}
where the summation is taken over $|\alpha_1|+|\beta_1| \leq
|\alpha|+|\beta|\leq [\tfrac{N}{2}]+4$, and $\alpha_2 \leq \alpha$
and $\beta_2 \leq \beta$ componentwise.
\end{lem}

\begin{lem}\label{Landau-Est}
For the Landau kernel, for any $l \geq 0$, there exist $C_{|\beta|},
C >0$, such that
\begin{equation}\label{linear-landau}
(w^{2l-2|\beta|}\pab \LL f\,,\pab f) \geq
\frac{1}{2}\|w^{l-|\beta|}\pab f\|^2_\sigma
-C_{|\beta|}\|f\|^2_\sigma\,,
\end{equation}
\begin{equation}\label{nonlinear-landau}
\begin{split}
(w^{2l-2|\beta|}\pab\Gamma(f,g)\,,\pab h) &\leq C\{\|w^{l-|\beta_1|}\p^{\beta_1}_{\alpha_1} f\|\cdot\|w^{l-|\beta_2|}\p^{\beta_2}_{\alpha_2} g\|_\sigma\\
&+\|w^{l-|\beta_1|}\p^{\beta_1}_{\alpha_1} g\|\cdot\|w^{l-|\beta_2|}\p^{\beta_2}_{\alpha_2} f\|_\sigma\}\\
&\times \|w^{l-|\beta|}\pab h\|_\sigma\,,
\end{split}
\end{equation}
where the summation is taken over $|\alpha|+|\beta|\leq N$, and
$\beta_1+\beta_2 \leq \beta$ and $\alpha_2 \leq \alpha$
componentwise.
\end{lem}

As a direct consequence of \eqref{nonlinear} in the above lemmas, we
can estimate the pure spatial derivatives for the nonlinear
collision operator $\Gamma$.

\begin{lem}
Let $\zeta(v)$ be a smooth function that decays exponentially, then
there is a given instant energy functional $\mathcal{E}_{N,0}(f)$
and $C_\zeta>0$, such that for summation over
$\alpha_1+\alpha_2=\alpha$, $|\alpha|\leq N$,
\begin{equation}\label{nonlinear2}
\begin{split}
(\p_\alpha\Gamma(f\,,g)\,,\p_\alpha h)&\leq\{\mathcal{E}^{1/2}_{N,0}(f)\|\p_{\alpha_2}g\|_\nu+\mathcal{E}^{1/2}_{N,0}(g)\|\p_{\alpha_2}f\|_\nu\}\|\p_{\alpha_3}h\|_\nu\,,\\
\left\|\int\p_\alpha\Gamma(f\,,g)\zeta\,dv\right\|&\leq
C_\zeta\{\mathcal{E}^{1/2}_{N,0}(f)\cdot\|\p_{\alpha_2}g\|_\nu+\mathcal{E}^{1/2}_{N,0}(g)\cdot\|\p_{\alpha_2}f\|_\nu\}\,.
\end{split}
\end{equation}
\end{lem}

\section{Derivation of Acoustic System}

In this section, we derive the acoustic system as the hydrodynamic
limit of solutions $f_\eps $ to the rescaled Boltzmann equation
\eqref{P_B}. Since we have the uniform energy bound in $\varepsilon$
 by Theorem \ref{Acoustic},
there exists the unique limit $f$ of $f_\eps $ in $\varepsilon$ and
we remark that due to higher order energy bound, all the limits in
the below are strongly convergent. First, by letting
$\varepsilon\rightarrow 0$ in \eqref{P_B}, one finds that
$\mathcal{L}f=0$. Thus $f$ can be written as follows:
\[
f = \{\rho+v\cdot
u+\left(\tfrac{|v|^2}{2}-\tfrac{3}{2}\right)\theta\}\sqrt{\mu},
\]
for $\rho,\,u,\,\theta$ are functions of $t,x$. In order to
determine the dynamics of $\rho,\,u,\,\theta$, project \eqref{P_B}
onto
$\{\sqrt{\mu},v\sqrt{\mu},(\tfrac{|v|^2}{2}-\tfrac{3}{2})\sqrt{\mu}\}$:
by collision invariants, first we get
\[
\langle \partial_t f_\eps +v\cdot \nabla_x f_\eps
,\;\{1,v,(\tfrac{|v|^2}{3}-1)\}\sqrt{\mu}\rangle=0
\]
and take the limit $\varepsilon\rightarrow 0$ to get
\[
\langle \partial_t f+v\cdot \nabla_x
f,\;\{1,v,(\tfrac{|v|^2}{3}-1)\}\sqrt{\mu}\rangle=0
\]
Since $f=\mathbf{P}f$, this is equivalent to
\begin{equation}\label{limit}
\begin{split}
&\partial_t\rho+\DIV u=0\\
&\partial_tu+\grad(\rho+\theta)=0\\
&\partial_t\theta+\tfrac{2}{3}\DIV u=0
\end{split}
\end{equation}
Thus we have shown the following proposition on the mathematical
derivation of the acoustic system from the Boltzmann equation.

\begin{prop}\label{convergence}
Assume that $F_\eps =\mu+\varepsilon\sqrt{\mu}f_\eps $
 solves the rescaled Boltzmann equation \eqref{R_B} where $f_\eps $
is obtained from Theorem \ref{Acoustic}. Then there exists the
hydrodynamic limit $f$ of $f_\eps $ such that $f=\mathbf{P}f$, and
furthermore its macroscopic variables $\rho,\,u,\,\theta$ solve the
acoustic system \eqref{limit}.
\end{prop}

The acoustic system is a linear system and it is globally well-posed
in the Sobolev space.

\begin{lem}\label{regularity} The acoustic system \eqref{limit} is  globally well-posed in  $H^s(\TT)$ space, for any $s\geq
0$. Moreover, we obtain the following estimates:
\begin{equation}\label{conser}
\frac{d}{dt}\{||\rho_1||^2_{H^s}+||u_1||^2_{H^s}+\tfrac{3}{2}||\theta_1||^2_{H^s}
\}=0
\end{equation}
\end{lem}

\begin{proof} The existence of solutions can be verified, for
instance by solving the ordinary differential equation after taking
Fourier transform in $x\in \TT$. The energy estimates give rise to
the conservation of energy \eqref{conser}. The uniqueness is easily
deduced.
\end{proof}

\section{uniform spatial energy estimates}

In this section, we shall establish a uniform spatial energy
estimate for $\fe$,
 a solution to \eqref{P_B}:
\[
 \partial_t f_\eps  +v\cdot\nabla_xf_\eps
+\frac{1}{\varepsilon} \mathcal{L}f_\eps =\Gamma(f_\eps , f_\eps )
\]
 For the convenience, we
rewrite the fluid part $\mathbf{P}f_\eps $ as follows:
\[
\mathbf{P}f_\eps =\{a^\varepsilon(t,x)+b^\varepsilon (t,x)\cdot
v+c^\varepsilon(t,x)|v|^2\}\sqrt{\mu}
\]
Our goal is to estimate $a^\varepsilon(t,x),\;b^\varepsilon
(t,x),\;\text{and }c^\varepsilon(t,x)$ in terms of
$(\mathbf{I-P})f_\eps $.

\begin{lem}\label{positivity} Assume $f_\eps $ is a solution
to \eqref{P_B} satisfying conservation of mass, momentum and energy:
\begin{equation}\label{conservation}
(f^{\varepsilon}(t),[1,v,|v|^2]\sqrt{\mu})=0.
\end{equation}
Then there exists $C_1>0$ such that
\begin{align}\label{PE}
\varepsilon\sum_{|\alpha|\leq N+1}\|\partial_\alpha\mathbf{P}f_\eps
\|^2 \leq \varepsilon\frac{dG(t)}{dt}+\frac{C_1}{\varepsilon}
\sum_{|\alpha|\leq N+1}\|\partial_\alpha(\mathbf{I-P})f_\eps
\|_\nu^2+ C_1\varepsilon \sum_{|\alpha|\leq
N}\|\partial_\alpha\Gamma(f_\eps , f_\eps )_{\|}\|^2
\end{align}
where $G(t)$ is defined as
\begin{equation}
\begin{split}
-&\sum_{|\alpha|\leq N}\int_{\TT}\left(\langle(\ip)\p_\alpha f_\eps
\,,\zeta_{ij}\rangle\cdot\p_j\p_\alpha b^\eps
-\langle(\ip)\p_\alpha f_\eps \,,\zeta_c\rangle\cdot\nabla_x\p_\alpha c^\eps\right)\,dx\\
&-\sum_{|\alpha|\leq N}\int_{\TT}\left(\langle(\ip)\p_\alpha f_\eps
\,,\zeta\rangle\cdot\nabla_x\p_\alpha a^\eps\,-\p_\alpha
b^\eps\cdot\nabla_x\p_\alpha a^\eps\right)\,dx\,.
\end{split}
\end{equation}
Here $\zeta_{ij}(v)\,,\zeta_c(v)\,,\zeta_a(v)$ are some fixed linear
combinations of the basis
\[
[\sqrt{\mu}, v_i\sqrt{\mu}, v_iv_j\sqrt{\mu},v_i|v|^2\sqrt{\mu}]
\]
for $1\leq i\,,j\leq 3$, and $f_{\|}$ is the $L^2_v$ projection of
$f$ onto the subspace generated by the same basis. It is obvious
that $|G(t)|\leq C \mathcal{E}_{N,0}(t) $.
\end{lem}
\begin{proof}
The proof is similar to the one of Lemma 6.1 in \cite{G}. For the
clear presentation of this article, we provide the key ingredients
and estimates and point out the difference.
 From the conservation of mass, momentum, and energy
\eqref{conservation}, it follows that
\[
\int_{\mathbb{T}^3}a^\varepsilon(t,x)dx=
\int_{\mathbb{T}^3}b^\varepsilon(t,x)dx=
\int_{\mathbb{T}^3}c^\varepsilon(t,x)dx=0
\]
By Poincar$\acute{e}$ inequality, it suffices to estimate
\[
\nabla_x\partial_\alpha a^\varepsilon,\;\nabla_x\partial_\alpha
b^\varepsilon,\;\nabla_x\partial_\alpha c^\varepsilon,\;
\]
for $|\alpha|\leq N$. First, we use the local conservation laws:
Multiply $\sqrt{\mu}, v\sqrt{\mu},|v|^2\sqrt{\mu}$ with \eqref{P_B}
and integrate in $v\in \mathbb{R}^3$. By the collision invariants,
we obtain
\begin{equation}\label{hydro}
\begin{split}
&\partial_ta^\varepsilon=\frac{1}{2}\langle v\cdot
\nabla_x(\mathbf{I-P})f_\eps ,|v|^2\sqrt{\mu}\rangle\\
&\partial_tc^\varepsilon+\frac{1}{3}\nabla_x\cdot
b^\varepsilon=-\frac{1}{6}\langle v\cdot
\nabla_x(\mathbf{I-P})f_\eps ,|v|^2\sqrt{\mu}\rangle\\
&\partial_tb^\varepsilon+\nabla_xa^\varepsilon+5\nabla_xc^\varepsilon=-\langle
v\cdot \nabla_x(\mathbf{I-P})f_\eps ,v\sqrt{\mu}\rangle
\end{split}
\end{equation}
The second ingredient of the proof is the macroscopic equations. By
plugging $f_\eps =\mathbf{P}f_\eps +(\mathbf{I-P}) f_\eps $ into
\eqref{P_B}, we get
\[
\begin{split}
\{\partial_ta^\varepsilon+\partial_tb^\varepsilon \cdot
v+\partial_tc^\varepsilon |v|^2\}\sqrt{\mu}+ v\cdot
\{\nabla_xa^\varepsilon +\nabla_xb^\varepsilon \cdot
v+\nabla_xc^\varepsilon|v|^2\}\sqrt{\mu}\\
=-\{\partial_t+v\cdot\nabla_x\} (\mathbf{I-P})f_\eps
-\frac{1}{\varepsilon}L(\mathbf{I-P})f_\eps + \Gamma(f_\eps ,f_\eps
)
\end{split}
\]
Fix $t,x$, and compare the coefficients on both sides in front of
\[
[\sqrt{\mu}, v_i\sqrt{\mu}, v_iv_j\sqrt{\mu},v_i|v|^2\sqrt{\mu}].
\]
Then we get the following macroscopic equations as
\begin{align}
\partial_ic^\varepsilon&=l_c^\varepsilon+h_c^\varepsilon\\
\partial_tc^\varepsilon+ \partial_i b_i^\varepsilon&=
l_{i}^\varepsilon+h_{i}^\varepsilon\\
\partial_i b_j^\varepsilon+\partial_j b_i^\varepsilon
&=l_{ij}^\varepsilon+h_{ij}^\varepsilon,\; i\neq j\\
\partial_tb^\varepsilon_i+ \partial_i a^\varepsilon&=
l_{bi}^\varepsilon+h_{bi}^\varepsilon\\
\partial_ta^\varepsilon &= l_a^\varepsilon+h_a^\varepsilon
\end{align}
Here the linear parts $l_c^\varepsilon,l_{i}^\varepsilon,
l_{ij}^\varepsilon,l_{bi}^\varepsilon,l_{a}^\varepsilon$ are of the
form
\begin{equation}
\langle -\{\partial_t+v\cdot\nabla_x\} (\mathbf{I-P})f_\eps
-\frac{1}{\varepsilon}\mathcal{L}(\mathbf{I-P})f_\eps , \zeta
\rangle
\end{equation}
where $\zeta$ is a linear combination of the basis
\[
[\sqrt{\mu}, v_i\sqrt{\mu}, v_iv_j\sqrt{\mu},v_i|v|^2\sqrt{\mu}],
\]
and accordingly,
$h_{c}^\varepsilon,h_{i}^\varepsilon,h_{ij}^\varepsilon
h_{bi}^\varepsilon,h_{a}^\varepsilon$ are  defined as $\langle
\Gamma(f_\eps ,f_\eps ),\zeta\rangle$ with the same choices of
$\zeta$.

Following the proof of Lemma 6.1 in \cite{G}, we first deduce
\begin{align*}
\|\nabla\partial_{\alpha}b^{\varepsilon}\|^{2} &\leq-\frac{d}{dt}
\int_{\mathbb{T}^{3}}\langle(\mathbf{I-P})
\partial_{\alpha
}f^{\varepsilon},\zeta_{ij}\rangle\cdot\partial_j\partial_{\alpha
}b^{\varepsilon}dx\\
&\;\;+C\|\nabla_{x}\partial_{\alpha}(\mathbf{I-P})
f^{\varepsilon}\|_{\nu}\{\|\nabla\partial_{\alpha}a^{\varepsilon}\|
+\|\nabla\partial_{\alpha}c^{\varepsilon}\|\}\\
&\;\;+C\{\|
\partial_{\alpha}(\mathbf{I-P})
f^{\varepsilon}\|_{\nu}^2+\|\nabla_{x}
\partial_{\alpha}(\mathbf{I-P})
f^{\varepsilon}\|_{\nu}^2\}
\\
&\;\;+\frac{C}{\varepsilon}\{\|\nabla_{x}\partial_{\alpha}(\mathbf{I-P})
f^{\varepsilon}\|_{\nu}+\|\partial_{\alpha}(\mathbf{I-P})f^{\varepsilon
}\|_{\nu}\}\|\nabla\partial_{\alpha}b^{\varepsilon}\|
+C\|\partial_{\alpha
}h_{\|}^{\varepsilon}\|\cdot\|\nabla\partial_{\alpha}b^{\varepsilon}\|.
\end{align*}
Note that the difference from the estimate in \cite{G} so far is
that the scaling parameter
 $\varepsilon$ is absent in the $t$-derivative term due to the
acoustic scaling. Now multiply it by $\varepsilon$ and apply the
Cauchy-Schwarz inequality to get
\begin{align*}
\varepsilon\|\nabla\partial_{\alpha}b^{\varepsilon}\|^{2} &
\leq-\varepsilon \frac
{d}{dt}\int_{\mathbb{T}^{3}}\langle(\mathbf{I-P})\partial_{\alpha
}f^{\varepsilon},\zeta_{ij}\rangle\cdot\partial_{j}\partial_{\alpha
}b^{\varepsilon}dx
+\frac{\varepsilon^2}{2}\{\|\nabla\partial_{\alpha}a^{\varepsilon}\|^{2}
 +\|\nabla\partial_{\alpha}c^{\varepsilon}\|^{2}\}\\
&\;\;+\frac{C}{\varepsilon}\{\|\nabla_{x}\partial_{\alpha}(\mathbf{I-P})%
f^{\varepsilon}\|_{\nu}^{2}+\|\partial_{\alpha}(\mathbf{I-P})f^{\varepsilon
}\|_{\nu}^{2}\} +C\varepsilon\|\partial_{\alpha
}h_{\|}^{\varepsilon}\|^2+\frac{\varepsilon}{2}
\|\nabla\partial_{\alpha}b^{\varepsilon}\|^{2}.
\end{align*}
By the same token, we obtain the similar estimates on
$\nabla\partial_\alpha c^\varepsilon$ and $\nabla\partial_\alpha
a^\varepsilon$ as follows:
\begin{align*}
\varepsilon\|\nabla\partial_{\alpha}c^{\varepsilon}\|^{2}  &  \leq-
\varepsilon\frac{d}{dt}%
\int_{\mathbb{T}^{3}}\langle(\mathbf{I-P})\partial_{\alpha
}f^{\varepsilon},\zeta_{c}\rangle\cdot\nabla_{x}\partial_{\alpha
}c^{\varepsilon}dx+\frac{\varepsilon^2}{2}
\|\nabla\partial_{\alpha}b^{\varepsilon}\|^{2}\\
&\;\;+\frac{C}{\varepsilon}\{\|\nabla_{x}\partial_{\alpha}(\mathbf{I-P})%
f^{\varepsilon}\|_{\nu}^{2}+\|\partial_{\alpha}(\mathbf{I-P})f^{\varepsilon
}\|_{\nu}^{2}\} +C\varepsilon\|\partial_{\alpha
}h_{\|}^{\varepsilon}\|^2+\frac{\varepsilon}{2}
\|\nabla\partial_{\alpha}c^{\varepsilon}\|^{2},\\
\varepsilon\|\nabla\partial_{\alpha}a^{\varepsilon}\|^{2}  & \leq
-\varepsilon\frac{d}{dt}\{\int_{\mathbb{T}^{3}}\langle(\mathbf{I-P})%
\partial_{\alpha}f^{\varepsilon},\zeta\rangle\cdot\nabla_{x}\partial_{\alpha
}a^{\varepsilon}dx+\int_{\mathbb{T}^{3}}\partial_{\alpha
}b^{\varepsilon}\cdot\nabla_{x}\partial_{\alpha}a^{\varepsilon}dx\}+
\frac{\varepsilon^2}{2}
\|\nabla\partial_{\alpha}b^{\varepsilon}\|^{2}\\
&\;\;+\frac{C}{\varepsilon}\{\|\nabla_{x}\partial_{\alpha}(\mathbf{I-P})%
f^{\varepsilon}\|_{\nu}^{2}+\|\partial_{\alpha}(\mathbf{I-P})f^{\varepsilon
}\|_{\nu}^{2}\} +C\varepsilon\|\partial_{\alpha
}h_{\|}^{\varepsilon}\|^2+\frac{\varepsilon}{2}
\|\nabla\partial_{\alpha}a^{\varepsilon}\|^{2}.
\end{align*}
By absorbing the hydrodynamic terms in the right hand sides into the
left hand sides, we obtain the desired estimates \eqref{PE}.
\end{proof}

Next we perform the energy estimates of spatial derivatives.

\begin{lem}\label{energy} Assume that $\fe$ is a solution to
equation \eqref{P_B} and satisfies \eqref{conservation}; then there
exists a constant $C_1\geq 1$ such that the following energy
estimate is valid:
\begin{equation}\label{energy-est}
\begin{split}
\frac{d}{dt}\{C_1&\sum_{|\alpha|\leq N+1}\|\partial_\alpha f_\eps
\|^2- \varepsilon\delta G(t)\}+\delta \sum_{|\alpha|\leq
N+1}\{\varepsilon\|
\partial_{\alpha}\mathbf{P}f_\eps \|^{2}
+\frac{1}{\varepsilon} \|\partial_{\alpha}(\mathbf{I-P})f_\eps
\|_{\nu}^{2}\}\\
&\leq 2C_1\sum_{|\alpha|\leq N+1}(\partial_\alpha\Gamma(f_\eps ,
f_\eps ),\partial_\alpha f_\eps ) +\varepsilon
\delta\sum_{|\alpha|\leq N}\|\partial_\alpha\Gamma(f_\eps ,
f_\eps )_{\|}\|^2\\
&\leq C\{\mathcal{E}_{N,0}^{1/2}(f_\eps ) +\mathcal{E}_{N,0}(f_\eps
)\}\mathcal{D}_{N,0}(f_\eps )
\end{split}
\end{equation}
\end{lem}
\begin{proof}
We take $\p_\alpha$ of \eqref{P_B} and sum over $\alpha$ to get
\begin{align*}
\frac{1}{2}\frac{d}{dt}\|\partial_\alpha f_\eps \|^2+
\frac{\delta}{\varepsilon}\|(\mathbf{I-P})\partial_\alpha f_\eps
\|_\nu^2 \leq (\partial_\alpha\Gamma(f_\eps , f_\eps
),\partial_\alpha f_\eps )\,.
\end{align*}
By Lemma \ref{positivity}, there is a constant $C_1\geq 1$ such that
\begin{equation}
\begin{split}
\frac{\delta}{2\varepsilon}\sum_{|\alpha|\leq N+1}&\|\p_\alpha\ipf\|^2_\nu\\
&\geq \frac{\delta\varepsilon}{2C_1}\sum_{|\alpha|\leq
N+1}\|\p_\alpha\pf\|^2-\frac{\delta\varepsilon}{2C_1}\frac{dG}{dt}-\frac{\delta\varepsilon}{2}
\sum_{|\alpha|\leq N}\|\p_\alpha\Gamma(\fe,\fe)_{\|}\|^2\,.
\end{split}
\end{equation}
Multiply by $C_1$ and collecting terms, we deduce the first
inequality in \eqref{energy-est}. By the nonlinear estimate in
\eqref{nonlinear2}, it is easy to derive that for $|\alpha|\leq N$
\begin{equation}\label{proj-est}
\|\p_\alpha\Gamma(f_\eps \,,f_\eps )_{\|}\|^2\leq
C\mathcal{E}_{N,0}(f_\eps )\mathcal{D}_{N,0}(f_\eps )\,,
\end{equation}
and
\begin{equation}
\begin{split}
(\p_\alpha\Gamma(f_\eps \,,f_\eps )\,,\p_\alpha f_\eps )&\leq C\E^{1/2}_{N,0}(f_\eps )\|\eps^{1/2}\pa f_\eps \|_\nu\|\eps^{-1/2}\pa\ipf\|_\nu\\
&\leq C\E^{1/2}_{N,0}(f_\eps )\D_{N,0}(f_\eps )\,.
\end{split}
\end{equation}
Thus, the second inequality in \eqref{nonlinear2} follows and this
finishes the proof of the lemma.
 \end{proof}

\section{The first order remainder}

In this section we finish the proof of Theorem \ref{Acoustic}. We
already established a pure spatial energy estimate for all collision
kernels in Lemma \ref{energy}. For general derivatives $\pab$,
different collision kernels require different weight functions, we
treat separately in two cases: hard potentials then soft potentials
and Landau kernel.

\subsection{Proof of hard potential case of Theorem \ref{Acoustic}}

\begin{proof} First note that
for the hydrodynamic part $\mathbf{P}f_\eps $,
\[
\|\partial_\alpha^\beta\mathbf{P}f_\eps \|\leq
C\|\partial_\alpha\mathbf{P}f_\eps \|
\]
which has been estimated in  Lemma \ref{energy}. In order to prove
Theorem \ref{Acoustic}, it remains to estimate the remaining
microscopic part $\partial_\alpha^\beta(\mathbf{I-P})f_\eps $ for
$|\alpha|+ |\beta|\leq N$. We take $\pab$ of equation \eqref{P_B}
and sum over $|\alpha|+ |\beta|\leq N$ to get
\begin{equation}\label{IP-equation}
\begin{split}
\p_t\pab&\ipf+v\cdot\nabla_x\pab\ipf+\frac{1}{\eps}\pab \mathcal{L}\ipf\\
&+\left(\p_t\pab\pf+\vgrad\pab\pf+\binom{\beta_1}{\beta}\p_{\beta_1}\vgrad\p^{\beta-\beta_1}_{\alpha}f_\eps \right)\\
&=\pab\Gamma(f_\eps \,,f_\eps )\,,
\end{split}
\end{equation}
where $|\beta_1|=1$. Taking the inner product with $w^{2l}\pab\ipf$,
we get
\begin{equation}\label{energy-id}
\begin{split}
\frac{d}{dt}&\left\{\frac{1}{2}\|\wl\pab\ipf\|^2\right\}+\frac{1}{\var}(w^{2l}\pab\mathcal{L}\ipf\,,\pab\ipf)\\
&+\left(\p_t\pab\pf+\vgrad\pab\pf+\binom{\beta_1}{\beta}\p^{\beta_1}_\alpha\vgrad\p^{\beta-\beta_1}_{\alpha}f_\eps
\,,w^{2l}\pab\ipf
\right)\\
&\leq\left(w^{2l}\pab\Gamma(f_\eps \,,f_\eps )\,,\pab\ipf\right)\,.
\end{split}
\end{equation}
By the linear estimate \eqref{linear}, we have
\begin{equation}
\frac{1}{\var}(w^{2l}\pab\mathcal{L}\ipf\,,\pab\ipf)\geq
\frac{1}{2\eps}\|\wl\pab\ipf\|^2_\nu-\frac{C}{\eps}\|\p_\alpha\ipf\|^2_\nu\,.
\end{equation}
From the local conservation laws \eqref{hydro} and the estimate
\eqref{proj-est},
\begin{equation}
\begin{split}
\|w^{2l}\p_t\pab\pf\|&\leq C\sum_{\aN}(\|\p_t\pa a^\eps\|+\|\p_t\pa b^\eps\|+\|\p_t\pa c^\eps\|)\\
&\leq C\left(\sum_{|\alpha|\leq N+1}\|\pa f_\eps \|_\nu+\sum_{\aN}\|\pa h^\eps_{\|}\|\right)\\
&\leq C\left(\sum_{|\alpha|\leq N+1}\|\pa f_\eps
\|_\nu+\E^{1/2}_{N,0}(f_\eps )\D^{1/2}_{N,0}(f_\eps )\right)\,.
\end{split}
\end{equation}
We also have
\begin{equation}
\|w^{2l}\vgrad\pab\pf\|\leq C\sum_{\aN}\|\nabla_x\pa\pf\|\leq
C\sum_{|\alpha|\leq N+1}\|\pa\pf\|\,.
\end{equation}
Thus the first two inner products in the second line of
\eqref{energy-id} is bounded by
\begin{equation}
\frac{1}{8\eps}\sum_{\abN}\|\pab\ipf\|^2_\nu+C\left(\sum_{|\alpha|\leq
N+1}\|\pa f_\eps \|^2_\nu+\E^{1/2}_{N,0}(f_\eps )\D_{N,0}(f_\eps
)\right)\,.
\end{equation}
The last term in the second line of \eqref{energy-id} is bounded by
\begin{equation}
\begin{split}
C&|(\p^{\beta_1}\vgrad\p^{\beta-\beta_1}_\alpha\ipf\,,w^{2l}\pab\ipf)|\\&+C|(\p^{\beta_1}\vgrad\p^{\beta-\beta_1}_\alpha\pf\,,
w^{2l}\pab\ipf)|\\
&\leq C\|\nabla_x\p^{\beta-\beta_1}_\alpha\ipf\|^2_\nu+\frac{1}{8\eps}\|w^l\pab\ipf\|^2_\nu+C \eps\|\pa\pf\|^2\\
&\leq C \eps\D_{N,l}(f_\eps )+\frac{1}{8\eps}\|w^l\pab\ipf\|^2_\nu+C
\eps\|\pa\pf\|^2\,,
\end{split}
\end{equation}
since $\nu(v)$ is bounded from below for hard potential.

Now we estimate the nonlinear term in \eqref{energy-id}. By the
nonlinear estimate in \eqref{nonlinear},
\begin{equation}
\begin{split}
(w^{2l}\pab\Gamma(f_\eps ,f_\eps )\,,\pab\ipf)&\leq C\E^{1/2}_{N,l}(f_\eps )\|\eps^{1/2}\wl\pa f_\eps \|_\nu\|\eps^{-1/2}\wl\pa\ipf\|_\nu\\
&\leq C\E^{1/2}_{N,l}(f_\eps )\D_{N,l}(f_\eps )\,.
\end{split}
\end{equation}
Using the coercivity of $\mathcal{L}$ \eqref{coer} and absorbing a
total of $\frac{1}{\eps}\|w^l\pab\ipf\|^2_\nu$ from the right-hand
side, we have
\begin{equation}\label{IP-est}
\begin{split}
\sum_{\abN}&\left(\frac{d}{dt}\left\{\frac{1}{2}\|\wl\pab\ipf\|^2\right\}+\frac{1}{4\eps}\|w^l\pab\ipf\|^2_\nu\right)\\
&\leq C\sum_{aaN}\|\pa f_\eps \|^2_\nu+C\left(\E^{1/2}_{N,l}(f_\eps
)+\eps\right)\D_{N,l}(f_\eps )\,.
\end{split}
\end{equation}

Multiplying \eqref{IP-est} by a factor and adding a large multiple
$K$ of \eqref{energy-est}, we have
\begin{equation}
\begin{split}
\frac{d}{dt}&(K\{C_1\sum_{\aaN}\|\pa f_\eps \|^2-\eps\delta
G(t)\}+2\sum_{\abN}\|w^l\pab\ipf\|^2)
+\D_{N,l}(f_\eps )\\
&\leq C_K\left(\E^{1/2}_{N,l}(f_\eps )+\E_{N,l}(f_\eps
)+\eps\right)\D_{N,l}(f_\eps )\,.
\end{split}
\end{equation}
Notice that
\begin{equation}
\|w^l\pab\pf\|^2\leq C\|\pa\pf\|^2\leq C\|\pa f_\eps \|^2\,,
\end{equation}
and
\begin{equation}
G(t)\leq C\sum_{aaN}\|\pa\pf\|(\|\ip\pa f_\eps \|+\|\pa \pf\|)\,.
\end{equation}
Thus we can redefine the instant energy by
\begin{equation}
\E_{N,l}(f_\eps )=K\{C_1\sum_{\aaN}\|\pa f_\eps \|^2-\eps\delta
G(t)\}+2\sum_{\abN}\|w^l\pab\ipf\|^2
\end{equation}
for $\eps$ sufficiently small. By a standard continuity argument, we
deduce our main estimate \eqref{en} by letting $\E_{N,l}(f_\eps )$
be sufficiently small initially.
\end{proof}

\subsection{Proof of soft potential and Landau cases for Theorem \ref{Acoustic}}

We follow the same idea as in the hard potential case to establish
\eqref{en} for both soft potential and Landau kernels. First, for
soft potential cases, we take inner product of
$w^{2(l-|\beta|)|\gamma|}\pab(\mathbf{I-P})f_\eps$ with the equation
\eqref{IP-equation} and sum over $|\alpha|+|\beta| \leq N$ to get
\begin{equation}\label{energy-id-soft}
\begin{split}
\frac{d}{dt}&\left\{\frac{1}{2}\|\wlbr\pab\ipf\|^2\right\}+\frac{1}{\var}(\wT\pab\mathcal{L}\ipf\,,\pab\ipf)\\
&+\left(\p_t\pab\pf+\vgrad\pab\pf+\binom{\beta_1}{\beta}\p^{\beta_1}_\alpha\vgrad\p^{\beta-\beta_1}_{\alpha}f_\eps
\,,\wT\pab\ipf
\right)\\
&\leq\left(\wT\pab\Gamma(f_\eps \,,f_\eps )\,,\pab\ipf\right)\,,
\end{split}
\end{equation}
for $|\beta_1|=1$. By the linear estimate \eqref{linear-soft}, we
have
\begin{equation}
\begin{split}
&\frac{1}{\var}(w^{2\{l-|\beta|\}|\gamma|}\pab\mathcal{L}\ipf\,,\pab\ipf)\\
&\geq
\frac{1}{2\eps}\|w^{\{l-|\beta|\}|\gamma|}\pab\ipf\|^2_\nu-\frac{C}{\eps}\|\p_\alpha\ipf\|^2_\nu\,.
\end{split}
\end{equation}
From the local conservation laws \eqref{hydro}, we have
\begin{equation}
\|w^{2\{l-|\beta|\}|\gamma|}\p_t\pab\pf\| \leq
C\left(\sum_{|\alpha|\leq N+1}\|\pa f_\eps
\|_\nu+\E^{1/2}_{N,0}(f_\eps )\D^{1/2}_{N,0}(f_\eps )\right)\,.
\end{equation}
We also have
\begin{equation}
\|w^{2\{l-|\beta|\}|\gamma|}\vgrad\pab\pf\|\leq
C\sum_{\aN}\|\nabla_x\pa\pf\|\leq C\sum_{|\alpha|\leq
N+1}\|\pa\pf\|\,.
\end{equation}
Note that $\|\cdot\|_\nu$ is equivalent to $\|w^{\gamma/2}\cdot\|$,
the first two inner products in the second line of
\eqref{energy-id-soft} is bounded by
\begin{equation}
\frac{1}{8\eps}\sum_{\abN}\|\wlbr\pab\ipf\|^2_\nu+C\left(\sum_{|\alpha|\leq
N+1}\|\pa f_\eps \|^2_\nu+\E^{1/2}_{N,0}(f_\eps )\D_{N,0}(f_\eps
)\right)\,.
\end{equation}
The weight function $w^{|\beta|\gamma}$ is so designed to treat the
last term in the second line of \eqref{energy-id-soft}
\begin{equation}
\begin{split}
C&|(\p^{\beta_1}\vgrad\p^{\beta-\beta_1}_\alpha\ipf\,,\wT\pab\ipf)|\\&+C|(\p^{\beta_1}\vgrad\p^{\beta-\beta_1}_\alpha\pf\,,
\wT\pab\ipf)|\\
&\leq C\|w^{l+|\beta-\beta_1|\gamma}\nabla_x\p^{\beta-\beta_1}_\alpha\ipf\|^2_\nu+\frac{1}{8\eps}\|\wlbr\pab\ipf\|^2_\nu+C \eps\|\pa\pf\|^2\\
&\leq C \eps\D_{N,l}(f_\eps
)+\frac{1}{8\eps}\|\wlbr\pab\ipf\|^2_\nu+C \eps\|\pa\pf\|^2\,.
\end{split}
\end{equation}
The nonlinear term in \eqref{energy-id-soft} is estimated by
\eqref{nonlinear-soft},
\begin{equation}
\begin{split}
&(\wT\pab\Gamma(f_\eps ,f_\eps )\,,\pab\ipf)\\
&\leq C\E^{1/2}_{N,l}(f_\eps )\|\eps^{1/2}\wlbr\pa f_\eps \|_\nu\|\eps^{-1/2}\wlbr\pa\ipf\|_\nu\\
&\leq C\E^{1/2}_{N,l}(f_\eps )\D_{N,l}(f_\eps )\,.
\end{split}
\end{equation}
The rest of the proof is similar to the hard potential case, the
nonlinear estimate \eqref{en} can be deduced by letting
\begin{equation}
\begin{split}
\E_{N,l}(f_\eps )&=K\{C_1\sum_{\aaN}\|\pa f_\eps \|^2-\eps\delta G(t)\}\\
&+2\sum_{\abN}\|\wlbr\pab\ipf\|^2\,.
\end{split}
\end{equation}
To establish the estimate \eqref{en} for the Landau case for which
the power of weight is $\gamma=-1$. We follow the same procedure as
in the soft potential case. Take the inner product with
$w^{2l-2|\beta|}\pab\ipf$ for equation \eqref{IP-equation} to get
\eqref{energy-id-soft} with $\gamma=-1$. All the estimates for the
soft potential case can applied for the case $\gamma=-1$. So we omit
the details here.\\

\indent{\bf Acknowledgements.} This project was initiated by the
suggestion of Prof. Yan Guo when N. Jiang visited The Applied
Mathematics Division of Brown University in May 2007. The main part
of this work was carried out when J. Jang was a member of IAS
2007-2008 academic year. The authors would like to appreciate the
hospitality of these institutes. They also wish to thank Y. Guo for
many discussions and his long-term support.

\end{document}